\documentclass{amsart}
\usepackage{extarrows}
\usepackage{tikz-cd}
\usepackage{graphicx}
\usepackage{amsmath,amssymb,amsfonts,amsthm,amscd}
\usetikzlibrary{matrix}
\usepackage{hyperref}

\newtheorem{theorem}{Theorem}[section]
\newtheorem*{theorem*}{Theorem}
\newtheorem{lemma}[theorem]{Lemma}
 \newtheorem{corollary}[theorem]{Corollary}
 \newtheorem{proposition}[theorem]{Proposition}
 \newtheorem*{mainthm*}{Main Theorem}

\theoremstyle{definition}

\theoremstyle{remark}
\newtheorem{remark}[theorem]{Remark}

\definecolor{alert}{rgb}{0.8,0,0.3}

\newcommand{\alert}[1]{%
	\marginpar{%
		\ifodd\value{page} \raggedright \else \raggedleft \fi
		\footnotesize{\textcolor{alert}{#1}}
	}
}
\newcommand{\E}{\mathbb{E}}
\newcommand{\R}{\mathbb{R}}
\newcommand{\Length}{\mathop{\rm Length}\nolimits}
\newcommand{\Prod}[1]{ \left\langle {#1} \right\rangle}
\newcommand{\prodesc}[2]{ \left\langle {#1},{#2} \right\rangle}

\begin{document}

\title[Parabolicity of invariant Surfaces]{Parabolicity of invariant Surfaces}

\author{Andrea Del Prete}
\address{Dipartimento di Ingegneria e Scienze dell'Informazione e Matematica, 
	Universit\`{a} degli Studi dell'Aquila,
	Via Vetoio Loc. Coppito -- 67100 L'Aquila, Italy}
\email{andrea.delprete@graduate.univaq.it}

\author{Vicent Gimeno i Garcia}
\address{Department of Mathematics, Universitat Jaume I-IMAC,   E-12071, 
Castell\'{o}, Spain}
\email{gimenov@mat.uji.es}

\thanks{The first author was partially supported by ``INdAM - GNSAGA Project", codice CUP\_E55F22000270001, and by MCIN/AEI project PID2022-142559NB-I00. The second author was partially supported by the Research grant  PID2020-115930GA-100 funded by MCIN/ AEI /10.13039/50110001103}




\dedicatory{}

\begin{abstract}
We present a clear and practical way to characterize the parabolicity of a complete immersed surface that is invariant with respect to a Killing vector field of the ambient space.
\end{abstract}

\maketitle

\section{Introduction} \label{sec:intro}
The intrinsic or extrinsic geometric conditions which guaranties that a surface is a parabolic surface has been largely studied from the beginning of the twentieth century. See for example \cite{Ahl},\cite{Mi} \cite{GriExp}, \cite{Hub},\cite{GimSol}. This interplay between geometry and analytical properties of functions defined on surfaces is used to classify the surfaces in conformal types.  Recall that a Riemannian manifold is said to be parabolic if every upper bounded subharmonic function is constant. A canonical reference on the parabolicity of manifolds is provided by \cite{GriExp}. In dimension $2$, parabolicity is a conformal property. Specifically, if the surface $(M,g)$ is parabolic, then the surface $(M,e^ug)$ is also parabolic for any smooth function $u:M\to \mathbb{R}$. Parabolic manifolds satisfy a specific maximum principle, as demonstrated in \cite{Alias2009}, for example. Given a non-constant $C^2$ function $u:M\to \mathbb{R}$ with $\sup_Mu=u^*<\infty$, if $M$ is a parabolic manifold, there exists a divergent sequence $\left\{x_k\right\}_{k\in \mathbb{N}}\subset M$ such that
$$
u(x_k)>u^*-\frac{1}{k},\quad {\rm and},\quad \Delta u(x_k)<0.
$$
Parabolicity can also be defined for manifolds $M$ with a non-empty boundary $\partial M$, as shown in \cite{Lopez20031017}. 
A manifold with a boundary is parabolic if and only if every bounded harmonic function is determined by its values on the boundary. In other words, if $f_1,f_2\colon M\to \mathbb{R}$ are two bounded harmonic functions with $f_1(x)=f_2(x)$ for any $x\in \partial M$, then $f_1(y)=f_2(y)$ for every $y\in M$. Given a complete Riemannian manifold $(M,g)$ and a precompact domain $\Omega\subset M$, every non-bounded connected component of $M-\Omega$ is called an end of $M$ with respect to $\Omega$. In this setting, $M$ is parabolic if and only if all the ends of $M$ with respect to $\Omega$ are parabolic.

In this article, we are interested in studying surfaces that are invariant with respect to a one-parameter group of isometries of the ambient space. For this purpose, we make use of Killing vector fields. A Killing vector field $\xi$ is defined such that $\mathcal{L}_\xi g=0$, where $g$ represents the metric tensor. We will say that a surface is invariant with respect to a Killing field when it is invariant with respect to the one-parameter group of isometries of the ambient manifold associated with $\xi$.

The main result of this paper is the following Theorem:

\begin{mainthm*}\label{MainThm}
    Let $\E $ be a $n$-dimensional Riemannian manifold which admits a complete Killing vector field $\xi\in\mathfrak{X}(\E)$. Assume that an immersed complete regular surface $S\subset\E$ is invariant with respect to the one-parameter group of isometries of $\E$ associated to $\xi$. Assume that there exists $q\in S$ such that $\Vert \xi(q)\Vert \neq0$  and denote by $\gamma(t)\subset S$ a complete curve parameterized by arc length satisfying \[\Prod{\Dot{\gamma}(t),\xi}=0.\]
    Then, $S$ is parabolic if and only if
    \begin{enumerate}
        \item the curve $\gamma$ is compact, or
        \item the integral curves of $\xi$ are compact and there exists $a,b\in\R$, with $a<b,$ such that $\int_{-\infty}^{a}\frac{1}{\|\xi(\gamma(s))\|}ds=\infty$ and $\int_{b}^\infty\frac{1}{\|\xi(\gamma(s))\|}ds=\infty$, or
        \item the integral curves of $\xi$ are non compact and $\int_{-\infty}^{0}\frac{1}{\|\xi(\gamma(s))\|}ds=\infty$ and $\int_{0}^\infty\frac{1}{\|\xi(\gamma(s))\|}ds=\infty.$
    \end{enumerate}
\end{mainthm*}

\begin{remark} We will see in Proposition \ref{prop:unua} that the existence of a point $p\in  S$ where $\xi(p)\neq 0$ allows us to prove that there exist at most two points where the Killing vector field $\xi$ vanishes. In particular, we will prove that, if $\xi_{|S}$ vanishes at exactly two points, then $S$ is diffeomorphic to a sphere and hence it is parabolic and, if $\xi_{|S}$ vanishes at only one point, then the integral curves of $\xi$ must be compact and since $\dot\gamma$ is perpendicular to $\xi$, there exist $a,b\in\R$, with $a<b$, such that $\|\xi(\gamma(s))\|\neq 0$ in $s\in(-\infty,b)\cup(a,\infty)$ and  $\int_{-\infty}^{a}\frac{1}{\|\xi(\gamma(s))\|}ds=\int_{b}^\infty\frac{1}{\|\xi(\gamma(s))\|}ds.$

Observe that if $\xi$ vanishes on all $S$, we cannot find a curve $\gamma\subset S$ which generates $S$, since every curve $\Gamma\subset S$ satisfies the condition $\Prod{\Dot{\Gamma}(t),\xi}=0$. However, this situation does not append in low dimensions. Indeed, \cite[Theorem 5.3]{KobaTrans} assures that each connected component of the set of zeros of $\xi$, ${\rm Zero}(\xi)$, is a closed totally geodesic submanifold of even codimension. In particular,if $n=3$, then $\xi$ vanishes at isolated geodesics of $\E$. If $n=4$, there could exist a surface $S$ where $\xi$ vanishes and its geometric properties does not depend on $\xi$ but only on the geometry of $\E$; for example, we can obtain a Killing vector field in $\mathbb{R}^4$ such that it vanishes identically at the totally geodesic submanifold $\mathbb{R}^2\subset \mathbb{R}^4$, and similarly we can find a Killing vector field in $\mathbb{H}^4$ such that it vanishes at the totally geodesic submanifold $\mathbb{H}^2\subset \mathbb{H}^4$ but $\mathbb{R}^2$ is parabolic and $\mathbb{H}^2$ is hyperbolic. If $n>4$, if $\xi$ vanishes on a surface $S\subset\E$, then $S$ can be immersed isometrically in a manifold that does not admit any Killing field a-priori.
\end{remark}
In what follows we prove the Main Theorem and then we show a couple of applications in some simply connected homogeneous manifold.

\section{Proof of the Main Theorem: Parabolicity of surfaces admiting a Killing submersion}
In order to prove the Main Theorem we need to state here that the surface only can admit finitely many points where the Killing vector field vanishes. In fact, $\xi$ vanishes at most on two points of $S$ and in such a case $S$ is diffeomorphic to a sphere. If $\xi$ only vanishes at one point of $S$, then $S$ is diffeomorphic to a plane. In both cases the integral curves of $\xi$ are diffeomorphic to $\mathbb{S}^1$, that is, $S$ is rotationally symmetric. Let us summarize this in the following proposition:
\begin{proposition}\label{prop:unua}Let $\E $ be a $n$-dimensional  Riemannian manifold which admits a complete Killing vector field $\xi\in\mathfrak{X}(\E)$. Assume that an immersed complete regular  surface $S\subset\E$ is invariant with respect to the one-parameter group of isometries of $\E$ associated to $\xi$. Assume that there exists $q\in S$ such that $\Vert \xi(q)\Vert \neq0$.
Then, there are only three options:
\begin{enumerate}
    \item The Killing vector field $\xi$ never vanishes on $S$.
    \item The Killing vector field $\xi$ vanishes at only one point $p\in S$. The surface $S$ is topologically a plane and the integral curves of $\xi$ are compact.
    \item The Killing vector field vanishes at only two points $p, p'\in S$. The surface $S$ is topologically an sphere and the integral curves of $\xi$ are compact.
\end{enumerate}

\end{proposition}
\begin{proof}
Assume that there exists a point $p\in S$ such that $\xi(p)=0$. First of all, we are going to prove that there are no other points with vanishing $\xi$ in a geodesic ball of $S$ centered at $p$ and radius  the injectivity radius of $p$.

Let us start joining $p$, where $\xi(p)=0$,  to the point $q$ (where is assumed that $\xi(q)\neq 0$) with a minimal geodesic segment $\gamma$. Assume that there exists a point $p'$, see figure \ref{Im1}, in $\gamma$ where $\xi(p')=0$.  
\begin{figure}[h!]
\centering
\includegraphics[scale=0.5]{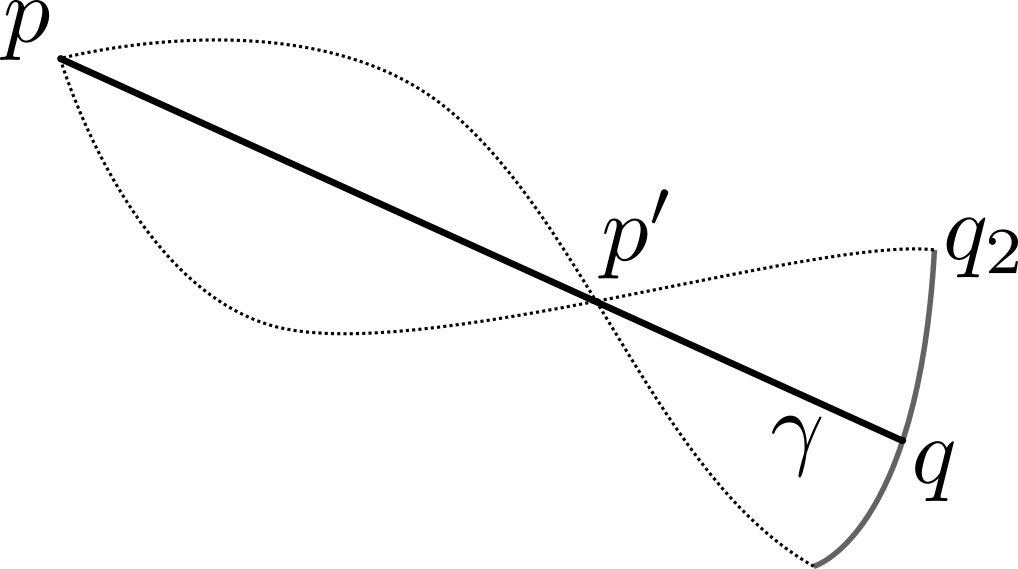}
\caption{If $\xi(p)=0$, $\xi(q)\neq0$ and $\gamma$ is the minimal geodesic segment joining $p$ to $q$, then there is no $p'\in\gamma$ with $\xi(p')=0$. }
\label{Im1}
\end{figure}

Let us denote by $\{\phi_t\}$ the one-parametric group of transformations associated to $\xi$. Let $q_2=\phi_{t_0}(q)$. Since $\xi(q)\neq 0$, we can take $t_0$ sufficiently small such that $q_2\neq q$. Moreover as we are assuming $\xi(p)=\xi(p')=0$, we know that $\phi_{t_0}(p)=p$ and $\phi_{t_0}(p')=p'$. Since $\xi$ is a Killing vector field, $\phi_t$ are isometries and $\gamma_1=\phi_{t_0}\gamma$ is a minimal geodesic segment as well joining $p$, $p'$ and $q_2$. By using the uniqueness of the geodesics, taking into account that $q_2\neq q$ we conclude $\gamma_1\neq \gamma$ and hence we have two different minimal geodesic segments joining $p$ and $p'$, then $p'\in {\rm cut}(p)$ and $\gamma$ does not realizes the distance providing a contradiction. Hence there are no points $p'$ in $\gamma$ with $\xi(p')=0$.

Therefore any point $p_1$ in $\phi_t(\gamma)$, with $0<t<t_0$, at distance less than the injectivity radius of $p$ satisfies that $\xi(p_1)\neq 0$. So, without loss of generality, we can assume that distance between $p$ and $q$ is smaller than the injectivity radius of $p$.

Take a point $q'$ in the minimal geodesic joining $p$ and $q$ at distance less than the injectivity radius of $p$. Then, $\xi(q')\neq 0$. Take the orbit the orbit by $\phi_t$ of $q'$,
$$
\mathcal{O}_{q'}=\left\{\phi_t(q')\, : \, t\in \mathbb{R} \right\}
$$
is a set of points of a geodesic sphere of radius ${\rm dist}(p,q')$ centered at $p$. Indeed, $\mathcal{O}_{q'}$ is  the geodesic sphere of distance ${\rm dist}(p,q)$ centered at $p$. Because otherwise  should there exists a point $p'$ of accumulation of $\mathcal{O}_{q'}$ in the geodesic sphere of distance ${\rm dist}(p,q)$ with $\xi(p')=0$ and  $\mathcal{O}_{q'}$
should be a subset of the metric sphere of distance ${\rm dist}(p',q)$ centered at $p'$ but this is a contradiction because $p'$ is a point of accumulation of $\mathcal{O}_{q'}$ and hence ${\rm dist}(p',\mathcal{O}_{q'})=0$. Therefore,
$$
\mathcal{O}_{q'}=\{ r \in S\, :\, {\rm dist}(p,r)={\rm dist}(p,q)\},
$$
which is compact and, as we have seen before, there are no points with $\xi=0$ in the minimal segments joining $p$ with the points of $\mathcal{O}_{q'}$.
In particular, we proved that, if $R$ is smaller, then the injectivity radius of $p$, then $\xi\neq0$ in $B_R(p)\setminus\left\{p\right\}$, where $B_R(p)$ is the geodesic ball centered at $p$ with radius $R$. Furthermore, $B_R(p)$ is foliated by integral curves of $\xi$. It can be proved that any point $p'$ with $\xi(p')=0$ then $p'\in {\rm cut}(p)$, see \cite[Corollary 5.2]{KobaTrans}. Moreover we can prove that if $q\in {\rm cut}(p)$ then $\xi(q)=0$. Suppose, by the contrary, that $\xi(p)=0$ and ${\rm cut}(p)\neq \emptyset$. Take a point $q\in {\rm cut}(p)$ let us consider the minimal geodesic segment 
$$
t\mapsto \gamma(t)=\exp_p(tu),\quad {\rm with}\quad \Vert u\Vert=1,
$$
such that $\gamma(0)=p$, $\gamma({\rm dist}(p,q))=q$ and $\gamma$ minimizes distance for $t\in [0,{\rm dist}(p,q)]$ but it is not a minimal geodesic for $t>{\rm dist}(p,q)$. If we suppose that $\xi(q)\neq 0$, the orbit $\mathcal{O}_q$  is the metric sphere $\partial B_{{\rm dist}(p,q)}(p)$, it has positive length and it is  diffeomorfic to $\mathbb{S}^1$. Moreover since $\phi_t$ acts by isometries every point   $q' \in \partial B_{{\rm dist}(p,q)}(p)$ is a cut point of $p$ and there exists $u'\in T_pM$ such that the geodesic segment 
$$
t\mapsto \gamma_1(t)=\exp_p(tu'),
$$
satisfies that $\gamma_1({\rm dist}(p,q))=q'$ and $\gamma_1$ minimizes distance for $t\in [0,{\rm dist}(p,q)]$ but it is not a minimal geodesic for  $t>{\rm dist}(p,q)$. Thence ${\rm cut}(p)=\partial B_{{\rm dist}(p,q)}(p)$ and, see for instance \cite[Lemma 4.4]{Sakai}, 
$$
S= B_{{\rm dist}(p,q)}(p)\cup \partial B_{{\rm dist}(p,q)}(p).
$$
But this is a contradiction with the completeness of $S$. Therefore, we can conclude that if $\xi(p)=0$ and $q\in {\rm cut}(p)$, then $\xi(q)=0$ and we can state that if $\xi(p)=0$,
$$
{\rm cut}(p)=\{ q\in S-\{p\}\, :\,  \xi(q)=0\}.
$$

Hence, if we suppose that $\xi(p)=0$ and ${\rm cut}(p)=\emptyset$,  $p$ is the only point in $S$ where the Killing vector field vanishes and $S$ is rotationally symmetric plane.

Now suppose again that $\xi(p)=0$ and ${\rm cut}(p)\neq \emptyset$. Take a point $p' \in {\rm cut}(p)$, and hence with $\xi(p')=0$, and such that ${\rm dist}(p,p')={\rm inj}(p)$. Thence for any point $q$ between $p$ and $p'$ such that $d(p',q)$ is smaller than the injectivity radius of $p'$, see figure \ref{Im2}, the orbit $\mathcal{O}_q$ need to be a geodesic sphere centered at $p$ and simultaneously a geodesic sphere centered at $p'$, namely,
$$
\mathcal{O}_q=\partial B_{{\rm dist}(p,q)}(p)=\partial B_{{\rm dist}(p',q)}(p').
$$
\begin{figure}[h!]
\centering
\includegraphics[scale=0.4]{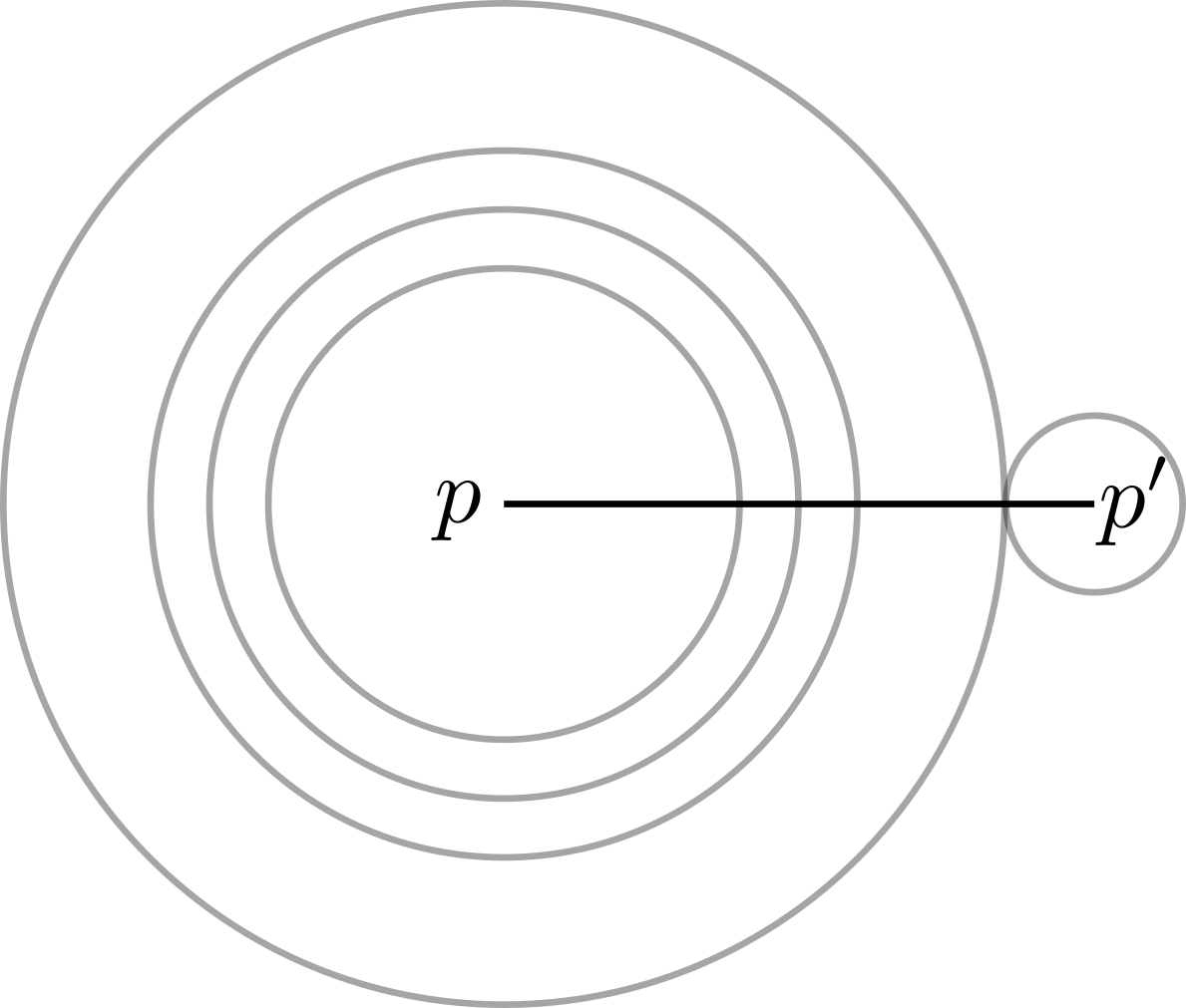}
\caption{If $S$ has two points $p$ and $p'$ where $ \xi$ vanish, then the points $p$ and $p'$ are antipodals points of a rotationally symmetric sphere.}
\label{Im2}
\end{figure}
This implies that $S$ is the connected sum
$$
S=\overline{B_{{\rm dist}(p,q)}(p)} \# \overline{B_{{\rm dist}(p',q)}(p')}.
$$
In particular, $S$ is a sphere and there exists no other point in $S$ where $\xi$ vanishes.

Summarizing everything, if we have a point $p\in S$ with $\xi(p)=0$ there are only two options: or ${\rm cut}(p)=\emptyset$ and $S$ is a rotationally symmetric plane without other points with vanishing $\xi$, or ${\rm cut}(p)$ contains only one point $p'$, where $\xi(p')=0$ and $S$ is a rotationally symmetric sphere. \end{proof}

By using the above proposition We can  remove a compact set $K$ of $S$ and the Killing vector field never vanishes on $S-K$. But $\Sigma$ is parabolic if and only if  $S-K$ is parabolic (see for instance \cite{GriExp}). Then, in order to simplify the discussion of the proof we are assuming that $\xi$ never vanishes on $S$, otherwise we can do the same argument but for $S-K$.

Since the surface $S$ is invariant by the one-parameter group of isometries $G_\xi=\{\phi_t\}$ associated  to the Killing vector field $\xi$, \emph{i.e.},
$
G_\xi(S)=S
$,
we will also assume that $G_\xi$  acts freely and properly on $S$, otherwise $G$ is not closed in $\mathrm{Iso}(S)$ with the compact-open topology and \cite[Proposition]{Lynge} implies that $\xi=X_1+X_2$ where $X_1$ and $X_2$ are two Killing vector fields with compact orbits satisfying $[X_1,X_2]=0$, that is, $S$ is a torus and hence it is a parabolic surface. Then, we can endow $S/G_\xi$ with a smooth structure and with a Riemannian metric tensor in such a way that the projection $\pi_2: S\to S/G_\xi$ is a Killing submersion as well.  Furthermore, since $\Prod{\Dot{\gamma}(t),\xi}=0$, the curve $\gamma$ is diffeomorfic to the set  $S/G_\xi$ and because we are using arc-length parametrization they are indeed isometric manifolds as well.

The Main Theorem is therefore equivalent to Theorem \ref{teo0} where we are proving the general case on surfaces which admits a Killing submersion. In order to prove this theorem we will prove Lemmas \ref{lemconf}, \ref{lemdua},\ref{lemtria},\ref{lemkvara}, and \ref{lemkvina}. In Lemma \ref{lemconf} we will prove that a conformal change on the metric tensor preserves the parabolicity for surfaces. Indeed, parabolicity is related to the conformal type in dimension $2$. In Lemma \ref{lemdua} we will prove that given a Killing submersion we can perform a conformal change with a basic function in the total space and in the base manifold in such a way that with respect to the new metric tensors the submersion remains as a Killing submersion. In Lemma \ref{lemtria} we will prove that a Killing submersion on a surface with constant norm of the Killing vector field has non-negative Gaussian curvature. In Lemma \eqref{lemkvara} we will prove that a complete surface with  non-negative Gaussian curvature  is a parabolic surface. Finally, with this previous lemmas and the technical lemma \ref{lemkvina} where is obtained the expression of the Laplacian of a basic function we can prove Theorem \ref{teo0} which is equivalent to the Main theorem. 

\begin{lemma}\label{lemconf}
In dimension $2$, parabolicity is preserved under conformal changes in the metric tensor.
\end{lemma}
\begin{proof}
    Let $(M,g)$ be a $2$-dimensional Riemannian manifold and consider the conformal change of metric 
    $\overline{g}=f^2g$
    given by the positive function $f:M\to \mathbb{R}_+$, then, the Laplacian $\overline{\Delta}$ with respect to the metric tensor $\overline{g}$ is related with the Laplacian $\Delta$ with respect to the metric tensor $g$ by (see \cite{Chavel} for instance)
    $$
\overline{\Delta}=\frac{\Delta}{f^2}.
    $$
    Then, $(M,g)$ admits non-constant bounded subharmonic functions if and only if $(M,f^2g)$ admits bounded subharmonic functions.
\end{proof}
\begin{lemma}\label{lemdua}
    Let $\pi:M\to B$ be a surjective Killing submersion from to the complete Riemannian manifold $(M,g_M)$ to $(B,g_B)$ with complete Killing vector field $\xi \in \mathfrak{X}(M)$. Let $f:B\to \mathbb{R}$ be a smooth and positive function. Then, $\pi:M\to B$ is also a Killing submersion from $(M,(f\circ \pi)^2g_M)$ to $(B,f^2g_B)$ with the complete Killing vector field $\xi$. Furthermore, 
    $(M,(f\circ \pi)^2g_M)$ is complete if $(B,f^2g_B)$ is complete.
\end{lemma}
\begin{proof}
    Since $\xi$ is a Killing vector field of $g_M$, the Lie derivative of the metric tensor vanishes, \emph{i.e.},
    $$
\mathcal{L}_\xi g_M=0.
    $$
    Then
    $$
\mathcal{L}_\xi\left((f\circ \pi)g_M\right)=\xi(f\circ \pi)\cdot g_M+(f\circ \pi)\mathcal{L}_\xi g_M=0,
    $$
    and hence, $\xi$ is also a Killing vector field for $(f\circ \pi)g_M$. Likewise, since $\pi:(M,g_M)\to (B,g_B)$ is a Riemannian submersion, for any $v\in T_pM$ with $g_M(v,\xi)=0$,
    $$
g_M(v,v)=g_B(d\pi(v),d\pi(v)).
    $$
Then, for any $v\in T_pM$ with $(f\circ \pi)g_M(v,\xi)=0$,
$$
(f\circ \pi)g_M(v,v)=f(\pi(p))g_B(d\pi(v),d\pi(v)),
$$
and hence $\pi:(M,(f\circ \pi)g_M)\to (B, fg_B)$ is a Riemannian submersion.

To prove that $M$ is complete, we consider an arbitrary Cauchy sequence $\left\{p_n\right\}_n$ in $M$ and prove that it is convergent.
We consider the sequence $\left\{q_n=\pi(p_n)\right\}_n\subset B.$
First notice that $\Prod{v,v}_M \geq\Prod{d\pi(v),d\pi(v)}_B$ for any point $p\in M$ and any tangent vector field $v\in T_pM.$ Then, $\Length_M(\gamma)\geq\Length_B(\pi(\gamma))$ for any curve $\gamma\subset M$. It follows that $\left\{q_n\right\}_n$ is a Cauchy sequence in $B$ and, since $BM$ is complete, $\left\{q_n\right\}_n$ converges to a point $q\in B$. In particular, we can assume that $\left\{q_n\right\}_n$ is contained in a compact and simply connected subset $K\subset B$ Let $F_0\colon K\to M$ be a local section, then, for any $n,$ there exists $t_n\in\R$ such that $p_n=\phi_{t_n}(q_n).$ 	
Denote by $c=\min_K\mu$. Then, for any $p\in\pi^{-1}(K)$ and any vector field $v\in T_pM$, we have $\prodesc{v}{v}_M\geq c\prodesc{d\pi^\bot(v)}{d\pi^\bot(v)}_\R$. This implies that $\|p_i-p_j\|_M\geq c|t_i-t_j|$ for any $i,j\in\mathbb{N},$ that is, $\left\{t_n\right\}_n$ is a Cauchy sequence in $(\R,g_{euc})$. Since $(\R,g_{euc})$ is complete, we can assume that there exist $a,b\in\R$ such that $t_n\in[a,b]$ for any $n$. It follows that $\left\{p_n\right\}_n$ is contained in the compact subset of $\pi^{-1}(K)$ delimited by $\phi_a(F_0)$ and $\phi_b(F_0).$
Hence, $\left\{p_n\right\}_n$ is a Cauchy sequence in a compact domain, that is convergent and this completes the proof.
\end{proof}

\begin{lemma}\label{lemtria}
    Let $\pi:S\to B$ be a Killing submersion with Killing vector field of constant norm. Then, if ${\rm dim}(S)=2$, $S$ has non-negative Gaussian curvature. 
\end{lemma}
\begin{proof}
    Given a point $p\in S$ and an horizontal vector $v\in \xi^\perp(p)$ with unit-length, $\Vert v\Vert=1$, in order to obtain the Gaussian curvature, \emph{i.e.}, the sectional curvature ${\rm sec}(v, \xi)$ of the plane spanned by $v$ and $\xi$, let us consider a vector field $\overline X\in\mathfrak{X}(B)$ defined in a neighborhood $U\ni\pi(p)$, such that $\overline X(\pi(p))=d\pi(v)$ and with vanishing covariant derivative $\nabla^B_{\overline X}\overline X=0$  in $B$, \emph{i.e.}, a geodesic vector field. Then, the lift $X\in \mathfrak{X}(S)$ of $\overline X$ defined in $\pi^{-1}(U) \ni p$ satisfies
\begin{equation}
\nabla_XX=(\nabla_XX)^H+\langle \nabla_XX,\xi\rangle\xi=-\langle X,\nabla_X\xi\rangle=0.
\end{equation}
Where here  the superscript $H$ denotes the horizontal part of a vector and we have used that since $\xi$ is a Killing vector fiels $\langle X,\nabla_X\xi\rangle=0$. Then,
$$
\begin{aligned}
{\rm sec}(v, \xi)=&\langle R(X,\xi)X,\xi\rangle=\langle \nabla_\xi\nabla_XX-\nabla_X\nabla_\xi X+\nabla_{[X,\xi]}X,\xi\rangle\\
=&\langle-\nabla_X\nabla_\xi X+\nabla_{[X,\xi]}X,\xi\rangle=\langle-\nabla_X\left([\xi, X]+\nabla_X\xi\right)+\nabla_{[X,\xi]}X,\xi\rangle\\
=&\langle\nabla_X\left([X,\xi]-\nabla_X\xi\right)+\nabla_{[X,\xi]}X,\xi\rangle
\end{aligned}
$$
In order to simplify the expression let us define the following vector fields $Y:=\nabla_X\xi$ and $Z:=[X,\xi]$.  Observe that both $X,Y$ are horizontal vector fields because $\xi$ has constant norm and thence $\langle \nabla_X\xi,\xi\rangle=\frac{1}{2}X\Vert \xi\Vert=0$. Moreover
$$
\langle Z,\xi\rangle=\langle \nabla_X\xi-\nabla_\xi X,\xi\rangle=\langle \nabla_X\xi,\xi\rangle-\langle\nabla_\xi X,\xi\rangle=\frac{1}{2}X\langle\xi,\xi\rangle=0
$$
Therefore,
$$
\begin{aligned}
{\rm sec}(v, \xi)=&\langle\nabla_X\left(Z-Y\right)+\nabla_{Z}X,\xi\rangle=\langle\nabla_XZ+\nabla_Z X,\xi\rangle-\langle\nabla_XY,\xi\rangle\\
=&\Vert Y\Vert^2\geq 0.
\end{aligned}
$$
where we have used
$$
\langle \nabla_XZ,\xi\rangle=-\langle Z,\nabla_X\xi\rangle=\langle X,\nabla_Z\xi\rangle=-\langle \nabla_ZX,\xi\rangle
$$
and
$
\langle\nabla_XY,\xi\rangle=-\langle Y,\nabla_X\xi\rangle=-\Vert Y\Vert^2
$.\end{proof}

\begin{lemma}\label{lemkvara}
    Let $S$ be a complete surface with non-negative Gaussian curvature. Then, $S$ is a parabolic manifold. 
\end{lemma}
\begin{proof}
   This lemma is a direct consequence of  a well known theorem due to Huber \cite{Hub} which states that if the negative part of the curvature $K_{-}=\max\{-K,0\}$ has finite integral, namely,
\begin{equation}\label{finite-negative}
\int_SK_{-}\,dA<\infty,
\end{equation}then, $\int_M K\, dA\leq \chi(M)$ and $M$ is conformally equivalent to a compact Riemann surface with finitely many punctures and hence it is a parabolic surface.\end{proof}
\begin{lemma}\label{lemkvina} Let $\pi:S\to \mathbb{R}$ be a Killing submersion with never vanishing Killing vector field $\xi\in \mathfrak{X}(M)$. Let $f:\mathbb{R}\to\mathbb{R}$ be a smooth function. Then
$$
\Delta^S(f\circ \pi)(q)=\frac{1}{\mu(x)}\frac{d}{dx}\left(\mu(x)\frac{df}{dx}\right)_{x=\pi(q)},
$$
  where  $\mu(x)$ is the norm of the Killing vector field $\Vert \xi(p)\Vert$ for any $p\in \pi^{-1}(x)$.
    \end{lemma}
    \begin{proof}
        Let $\pi: (M,\langle,\rangle_M)\to (B,\langle,\rangle_B)$ be a Rieamannian submersion. The Laplacian of a basic function ($\widetilde f= f\circ \pi,\quad  f:B\to \mathbb{R}$) is given by (see \cite[Lemma 3.1]{Bessa2012} for instance) 
\begin{equation}
  (\triangle^M \widetilde f)_q=(\triangle^B f)_x-\langle\nabla^B  f_p,\, d\pi_q(H_q)\rangle_B,
\end{equation}
where $x=\pi(q)$, and $H$ is the mean curvature vector field of the fiber $\pi^{-1}(p)$.  In our case taking $B=\mathbb{R}$ and $M=S$ this implies
\begin{equation}\label{eq:kvara}
  (\triangle^S \widetilde f)_q=\left.\frac{d^2f}{dx^2}\right\vert_{x=\pi(q)}-\left.\frac{df}{dx}\right\vert_{x=\pi(q)}\left\langle\frac{\partial}{\partial x},\, d\pi_q(H_q)\right\rangle_\mathbb{R},
\end{equation}
where $\frac{\partial}{\partial x}$ is the unit vector tangent to $\mathbb{R}$. Taking into account that the mean curvature vector field is given by
$$
H=\nabla_{\frac{\xi}{\mu}}\frac{\xi}{\mu}=\frac{1}{\mu^2}\nabla_\xi \xi.
$$
Then, 
\begin{equation}
\begin{aligned}
    \left\langle\frac{\partial}{\partial x},\, d\pi_q(H_q)\right\rangle_\mathbb{R}=&\frac{1}{\mu^2}\left\langle\frac{\partial}{\partial x},\, d\pi(\nabla_\xi \xi)\right\rangle_\mathbb{R}=\frac{1}{\mu^2}\left\langle X,\, \nabla_\xi \xi\right\rangle\\
    =&\frac{-1}{\mu^2}\left\langle \xi,\, \nabla_X \xi\right\rangle=\frac{-1}{2\mu^2}X\Vert\xi\Vert^2=\frac{-1}{2\mu^2}\frac{d}{dx}\mu^2
\end{aligned}    
\end{equation}
where $X$ is an horizontal lift of $\frac{\partial}{\partial x}$ and we have used that $\xi$ is a Killing vector field. Finally equation \eqref{eq:kvara} becomes
 \begin{equation}
  (\triangle^S \widetilde f)_q=\left.\frac{d^2f}{dx^2}\right\vert_{x=\pi(q)}+\left.\frac{df}{dx}\right\vert_{x=\pi(q)}\frac{1}{\mu(x)}\left.\frac{d\mu}{dx}\right\vert_{x=\pi(q)},
\end{equation}
and the proposition is proved.\end{proof}

By using the previous lemmas,    instead of study the paraboilcity of a complete surface $(S,g)$ which admits a Killing submersion  $\pi:S\to M^1$  to the connected $1$-dimensional Riemannian manifold $(M,g_{\rm can})$, we will study the parabolicity of the conformally equivalent Riemannian manifold $(S, \frac{1}{\mu^2}g)$ where $\mu(x)$ is the norm of the never vanishing Killing vector field $\Vert \xi(p)\Vert$ for any $p\in \pi^{-1}(x)$. By the lemma \ref{lemconf} $(S,g)$ is parabolic if and only if $(S, \frac{1}{\mu^2}g)$ is parabolic.  Observe that since $S$ is complete $\pi$ needs to be a surjective map. Moreover by lemma \ref{lemdua}, the map $\pi$ induces a Riemannian submersion from $(S, \frac{1}{\mu^2}g)$ to $(M,\frac{1}{\mu^2}g_{\rm can} )$ . Since this submersion has a Killing vector field of constant norm,  $(S, \frac{1}{\mu^2}g)$ is a surface with non-negative Gaussian curvature  by lemma \ref{lemtria}. When $M=\mathbb{S}^1(R)$ or $M=\mathbb{R}^1$ and 
    $$
\int_{-\infty}^0 \frac{dx}{\mu(x)}=\infty\quad{\rm and}\quad \int_0^\infty \frac{dx}{\mu(x)}=\infty.
$$
we have that $(M,\frac{1}{\mu^2} g_{\rm can})$ is complete. Hence by using lemma \ref{lemdua}  $S$ is a complete surface with non-negative Gaussian curvature and thus by lemma \ref{lemkvara}, the surface $S$ is a parabolic manifold.   On the other hand, if we assume that
$$
\int_{-\infty}^0 \frac{dx}{\mu(x)}<\infty\quad\left({\rm or}\quad \int_0^\infty \frac{dx}{\mu(x)}<\infty\right).
$$
By lemma \ref{lemkvina} the  function $F:M\to \mathbb{R}$
$$
F(p):=\int_{\pi(p)}^0 \frac{dx}{\mu(x)}
$$
is a bounded and harmonic function (which implies that $S$ is a non parabolic manifold). This can sumarized in the following theorem which implies the Main theorem of the parer:

\begin{theorem}\label{teo0}Let $\pi:S\to M^1$ be a Killing Submersion from a complete and $2$-dimensional Riemannian manifold $(S,g)$ to the connected $1$-dimensional Riemannian manifold $(M,g_{\rm can})$. Let us denote by $\mu(x)$ the norm of the Killing vector field $\Vert \xi(p)\Vert$ for any $p\in \pi^{-1}(x)$. Then, 
\begin{enumerate}
    \item If $M=\mathbb{S}^1(R)$ endowed with the canonical metric tensor, $S$ is parabolic.
    \item If $M=\mathbb{R}$ with its canonical metric tensor, $S$  is a parabolic manifold iff
$$
\int_{-\infty}^0 \frac{dx}{\mu(x)}=\infty\quad{\rm and}\quad \int_0^\infty \frac{dx}{\mu(x)}=\infty.
$$
\end{enumerate}    
\end{theorem}

\section{Application of the Main Theorem: Parabolicity of invariant surfaces in homogeneous 3-manifolds}
In this section we use \textbf{Main Theorem} to study the parabolicity of invariant surfaces with some geometric properties in 3-dimensional homogeneous Riemannian manifold studied in different works by many authors.

\subsection{Parabolicity of invariant surfaces with constant mean curvature in $Sol_3$}
In \cite{lopez2014invariant}, Lopez and Munteanu give a description of invariant surfaces with either constant curvature in the Thurston geometry $Sol_3$. Recall that $Sol_3$ is isometric to $\R^3$ endowed with the metric \[e^{2z} dx^2+e^{-2z}dy^2+dz^2\] and the component of the identity its isometry group is generated by the following families of isometries:\\
\resizebox{\hsize}{!}{$T_1^c(x,y,z)=(x+c,y,z),\quad T_2^c(x,y,z)=(x,y+c,z),\quad T_3^c(x,y,z)=(e^{-c}x,e^{c}y,z+c).$}
\\
In particular, in \cite{lopez2014invariant} they studied surfaces with constant mean curvature or constant Gaussian curvature invariant with respect to $T_1$, finding a profile curve $\gamma(s)=(0,y(s),z(s))$ parameterized by arc length. Notice that the Killing vector field $\partial_x$ associated to $T_1$ has norm $\|\partial_x\|=e^z$ and it is orthogonal to $\gamma'$, so we can easely apply our result just studyng $\int e^{-z(s)}ds$. 

The classification of minimal surfaces \cite[Theorem 3.1]{lopez2014invariant} assures that the only $T_1$-invariant minimal surfaces of $Sol_3$ are:
\begin{enumerate}
    \item a leaf of the foliation $\left\{Q_t=\left\{(x,t,z)\mid \, x,z\in\R\right\}\right\}_{t\in\R}$, which are known to be isometric to the hyperbolic plane;
    \item a leaf of the foliation $\left\{R_t=\left\{(x,y,t)\mid \, x,y\in\R\right\}\right\}_{t\in\R}$, which are known to be isometric to the Euclidean plane;
    \item the surfaces $S_{\theta_0,a}$, for $\theta_0\in\R\setminus\left\{k\pi:\,k\in\mathbb{Z}\right\}$ and $a\in\R$, that are generated by translating the profile curve $$\gamma(s)=\left(0,a+ e^{s\sin(\theta_0)},\log\left(\tan(\theta_0)e^{s\sin(\theta_0)}\right)\right).$$
\end{enumerate}
Obviously, the surfaces $Q_t$ are hyperbolic, while the surfaces $R_t$ are parabolic. It remains to study the parabolicity of $S_{\theta_0,a}.$ Since 
\[\int e^{-\log\left(\tan(\theta_0)e^{s\sin(\theta_0)}\right)}ds=e^{-s\sin(\theta_0)}\cot(\theta_0)\csc(\theta_0),\]
is a bounded function, \textbf{Main Theorem} implies that $S_{\theta_0,a}$ are hyperbolic surfaces. 

For the constant mean curvature case, \cite[Theorem 3.2]{lopez2014invariant} proves that the $z$-coordinate of the profile curve is bounded. In particular, \[\lim_{s\to\pm\infty}\int e^{-z(s)}ds=\pm\infty\] and \textbf{Main Theorem} implies that the respective invariant $H$-surfaces are parabolic.

\subsection{Parabolicity of vertical cylinders in Killing submersion}
An easy way to study surface surfaces immersed in a three-manifold which are invariant with respect to a Killing vector field of the ambient space is by studying vertical cylinders in Killing submersions.
We recall that a Killing submersion is a Riemannian submersion $\pi\colon\E\to M$ from a three-dimensional manifold $\E$ onto a surface $(M,g)$, both connected and orientable, such that the fibers of $\pi$ are integral curves of a Killing vector field $\xi\in\mathfrak{X}(\E).$ These spaces have been completely classified in terms of the Riemannian surface $(M,g)$, the length of the Killing vector field $\mu=\|\xi\|$ and the so called bundle curvature $\tau,$ defined such that $\tau(p)=\frac{-1}{\mu(p)}\langle\overline\nabla_u\xi,v\rangle$, where $\{u,v,\xi_p/\mu(p)\}$ is a positively oriented orthonormal basis of $T_p\E$ (see \cite[Section 2]{LerMan17}). In this setting a vertical cylinder is a surface $S$ always tanget to $\xi$. In particular, $S$ is invariant with respect to the isometries associated to $\xi$ and it projects through $\pi$ onto a curve $\gamma\subset M$. Since $\pi$ is a Riemannian submersion, to study $\int_\Gamma\frac{1}{\mu(\Gamma)}$, where $\Gamma\subset S$ is a curve parameterized by arc length and orthogonal to $\xi$, is equivalent to study $\int_\gamma\frac{1}{(\pi_* \mu)(\gamma)}$ where $\gamma=\pi(S)$ is parameterezid by arc length in $(M,g)$. In particular, in this setting, the Main Theorem, read as follows.

\begin{theorem}
    Let $\pi\colon \E\to (M,g)$ be a Killing submersion with Killing length $\mu$. Then, for any complete curve $\gamma\subset M$, $\pi^{-1}(\gamma)$ is a parabolic surface if and only if $\gamma$ is complete in $(M,\tfrac{1}{\mu^2}g)$.    
\end{theorem}
\begin{remark}
    It is interesting to notice that, given a curve $\gamma\subset M$, while the conformal metric $\tfrac{1}{\mu^2}g$ gives us information about the parabolicity of $S=\pi^{-1}$, the conformal metric $\mu^2 g$ gives us information about its mean curvature (see \cite[Proposition 2.3]{DMN}).
\end{remark}

The parabolicity of surfaces of revolution (or surfaces with ends of revolution) in spaces forms of constant sectional curvature has been studied in \cite{Gimeno2016143}. Three dimensional simply connected spaces forms of constant sectional curvature are the only manifolds with a group of isometries of dimension $6$ and can be classified (up to isometries) by their sectional curvatures.  We can label these spaces hence as $\mathbb{M}^3(\kappa)$. The spaces $\mathbb{M}^3(\kappa)$ includes the Euclidean space for $\kappa=0$, the spheres for $\kappa>0$ and the Hyperbolic space for $\kappa<0$.

There are no three dimensional spaces with a group of isometries of dimension $5$. The simply-connected $3$-dimensional spaces with a group of isometries of dimension $4$ are Killing submersions and they can be classified (up to isometries) by the curvature $\kappa$ of the base manifold and by the torsion $\tau$ of the fibers.   We can label these spaces hence as $\mathbb{E}^3(\kappa,\tau)$ (because they are endowed with the Riemannian submersion $\pi:\mathbb{E}^3(\kappa,\tau)\to \mathbb{M}^2(\kappa)$ with constant bundle curvature $\tau$). In what follows we give a general condition to guarantee the parabolicity of rotational surfaces in the $\mathbb{E}^3(\kappa,\tau)$-spaces.

The canonical rotational model describing the $\E^3( \kappa,\tau )$-spaces is given by $(\Omega\times\R,ds^2)$, where
\[ \begin{array}{l}
      \Omega=\left\{(x,y)\in\R^2\mid \lambda(x,y)>0\right\},\qquad
      \lambda(x,y)=\frac{1}{1+\tfrac{\kappa}{4}(x^2+y^2)},\\
    ds^2=\lambda^2(x,y) (dx^2+dy^2) + (\lambda(x, y)\tau (y dx - x dy) + dz)^2.
\end{array} \]
Since the Killing vector field $\|\partial_z\|$ has unitary norm, every surface that is invariant with respect to $\partial_z$ is parabolic. So, we focus on studying the parabolicity rotational surfaces.
First notice that if $\kappa>0$, $\E(\kappa,\tau)$ is a Berger sphere, in particular, since it is compact, every Killing vector field $\xi$ has bounded norm and from the Main Theorem we deduce that every surface that is invariant with respect to $\xi$ is parabolic.
When $\kappa\leq0$, we consider the Killing vector field $\xi(x,y)=-y \partial_x +x \partial_y$ generating the rotation around the $z$-axis and to describe $\E(\kappa,\tau)$ as a Killing submersion with respect to $\xi$, we use the cylindrical coordinates 
\[x(r,\theta)=r\cos(\theta),\qquad y(r,\theta)=r\sin(\theta).\]
We obtain that the space $\E(\kappa,\tau)$ minus its $z$-axis is isometric to a quotient of $(\R^2_\kappa\times\R,ds^2_{rot})$, where $\R^2_\kappa=\left\{(r,z)\in\R^2|r>0,\,4+\kappa r^2>0\right\}$ and 
\begin{equation*}
    ds^2_{rot}=\frac{16 dr^2}{(4 + \kappa r^2)^2} + \frac{dz^2}{(1 + r^2 \tau^2)}
    + \left(\frac{2r \sqrt{1 + r^2 \tau^2} }{
    4 + \kappa r^2}\right)^2 \left(d\theta -  \frac{(4+\kappa r^2)\tau}{4(1+r^2\tau^2)}dz\right)^2.
\end{equation*}
Here, $\xi=\partial\theta$ and the Killing submersion with respect to $\xi$ is such that
\[(M,g)=\left(\R^2_\kappa,\frac{16}{(4 + \kappa r^2)^2} dr^2 + \frac{1}{(1 + r^2 \tau^2)}
   dz^2\right)\quad \textrm{ and }\quad \mu(r,z)=\frac{2r \sqrt{1 + r^2 \tau^2} }{
    4 + \kappa r^2}\]
and then the conformal metric tensor $g/\mu^2$ is given by
\begin{equation}\label{eq:unua}
\frac{1}{\mu^2}g=\frac{4}{r^2(1+r^2\tau^2)}dr^2+\frac{(4+\kappa r^2)^2}{4r^2(1+r^2\tau^2)^2}dz^2.  
\end{equation}

In particular, every complete curve $\gamma$ in $(M,g)$ generetes a complete parabolic invariant surface in $\mathbb{E}^3(\kappa,\tau)$ if and only if $\gamma$ is complete with respect Using the metric \eqref{eq:unua}. For example, using this tool, it is easy to see that the minimal umbrellas of the Heisenberg group $\E^3(0,\tau)$ are hyperbolic, without computing their extrinsic area growth (see \cite{ManNel17}). It is sufficient to consider the curve $\gamma(t)=(t,0)$. Its norms is in the conformal metric is $\|\gamma(t)\|=\frac{2}{t\sqrt{1+t^2\tau^2}}<t^{-3/2}$ for $t>\frac{2+\sqrt{4-\tau^2}}{\tau^2}$ or $\tau>2$, that is, $\gamma$ is not complete in the conformal metric, thus $\pi^{-1}(\gamma)$ is hyperbolic.

We can also use this tool to study the parabolicity of the rotational surfaces of constant mean curvature in $\E^3(-1,\tau)$ described in \cite{Pe}. In particular, Peñafiel shows that a rotational surface of constant mean curvature $H\in\R$ is parameterized by $\gamma_d(t)=\left(\tanh\left(\tfrac{\sqrt{t}}{2}\right),u_d(t)\right),$ where 
\[u_d(t)=\int \frac{(2H\cosh(r)+d)\sqrt{1+4\tau^2\tanh^2\left(\tfrac{r}{2}\right)}}{\sqrt{\sinh^2(r)-(2H\cosh(r)+d)^2}},\]
with $d\in\R.$
When $d=-2H$, the rotation of $\gamma_d(t)$ generates an entire graph.
When $H=0$, the norm of $\gamma'(t)$ with respect to \eqref{eq:unua} is 
\[\|\gamma'(t)\|=2\sqrt{\frac{1}{\sinh^2(t)\left(1+\tau^2\tanh\left(\tfrac{t}{2}\right)\right)}}\simeq 4\sqrt{\frac{1}{1+\tau^2}}e^{-t}+o(e^{-2t}).\]
In particular, $\lim_{t\to\infty}\int\|\gamma'(t)\|$ is convergent and the rotational end generated by the rotation of $\gamma$ is hyperbolic. Furthermore, since the difference between $u_d$ and $u_{-2H}$ is bounded, the surface generated by rotating any $\gamma_d$ is hyperbolic. 

On the contrary, when $H=1/2$, we get that 
\begin{center}
    \resizebox{\hsize}{!}
{$
\|\gamma'(t)\|=\left(\frac{(5+3\cosh(t))^2(1-4\tau^2+\cosh(t)+4\tau^2\cosh(t))}{8(1-\tau^2+(1+\tau^2)\cosh(t))^2}+\frac{4}{\sinh^2(t)\left(1+4\tau^2\tanh\left(\tfrac{t}{2}\right)\right)}\right)^{\tfrac{1}{2}}
$}
\end{center}
which diverges for $t\to+\infty$. That is, the entire rotational graph with critical constant mean curvature is parabolic.

\def\cprime{$'$} \def\cprime{$'$}
\def\polhk#1{\setbox0=\hbox{#1}{\ooalign{\hidewidth
			\lower1.5ex\hbox{`}\hidewidth\crcr\unhbox0}}}
\def\polhk#1{\setbox0=\hbox{#1}{\ooalign{\hidewidth
			\lower1.5ex\hbox{`}\hidewidth\crcr\unhbox0}}}
\def\polhk#1{\setbox0=\hbox{#1}{\ooalign{\hidewidth
			\lower1.5ex\hbox{`}\hidewidth\crcr\unhbox0}}} \def\cprime{$'$}
\def\cprime{$'$} \def\cprime{$'$} \def\cprime{$'$} \def\cprime{$'$}

\end{document}